\documentclass[11pt]{amsart} \usepackage{latexsym, amssymb, stmaryrd}
\usepackage[T1]{fontenc}

\newtheorem{thm}{Theorem}

\newtheorem{prop}[thm]{Proposition}
\newtheorem{cor}[thm]{Corollary}

\theoremstyle{definition}

\newtheorem{rmk}[thm]{Remark}

\theoremstyle{remark}

\makeatletter

\def\dotminussym#1#2{%
  \setbox0=\hbox{$\m@th#1-$}%
  \kern.5\wd0%
  \hbox to 0pt{\hss\hbox{$\m@th#1-$}\hss}%
  \raise.6\ht0\hbox to 0pt{\hss$\m@th#1.$\hss}%
  \kern.5\wd0}

\newcommand{\Z}{\mathbb{Z}}

\newcommand{\curly}[1]{\mathcal{#1}}

\newcommand{\B}{\curly{B}}

\renewcommand{\to}{\rightarrow}

\def \<{\langle}
\def \>{\rangle}

\def \*Z {{{^*}\Z}}
\def \((  {(\!(}
\def \)) {)\!)}

\numberwithin{equation}{section}

\def \id{\operatorname{id}}

\def\om{\omega}

\allowdisplaybreaks[2]

\title{Model theory and the QWEP Conjecture}
\author{Isaac Goldbring}
\thanks{The work here was partially supported by NSF CAREER grant DMS-1349399.}
\address {Department of Mathematics, Statistics, and Computer Science, University of Illinois at Chicago, Science and Engineering Offices M/C 249, 851 S. Morgan St., Chicago, IL, 60607-7045}
\email{isaac@math.uic.edu}
\urladdr{http://www.math.uic.edu/~isaac}

\begin{document}

\begin{abstract}
We observe that Kirchberg's QWEP conjecture is equivalent to the statement that $C^*(\mathbb{F})$ is elementarily equivalent to a QWEP C$^*$ algebra.  We also make a few other model-theoretic remarks about WEP and LLP C$^*$ algebras.
\end{abstract}
\maketitle

For the sake of simplicity, all C$^*$ algebras in this note are assumed to be unital.

Suppose that $B$ is a C$^*$ algebra and $A$ is a subalgebra.  We say that $A$ is \emph{relatively weakly injective} in $B$ if there is a u.c.p.\ map $\phi:B\to A^{**}$ such that $\phi|B=\id_A$; such a map is referred to as a \emph{weak conditional expectation}.  (We view $A$ as canonically embedded in its double dual.)  A C$^*$ algebra $A$ is said to have the \emph{weak expectation property} (or be WEP) if it is relatively weakly injective in every extension and $A$ is said to be \emph{QWEP} if it is the quotient of a WEP algebra.

\emph{Kirchberg's QWEP Conjecture} states that every separable C$^*$ algebra is QWEP.  In the seminal paper \cite{Kirchberg}, Kirchberbg proved that the QWEP Conjecture is equivalent to the Connes Embedding Problem (CEP), namely that every finite von Neumann algebra embeds into an ultrapower of the hyperfinite II$_1$ factor.  

If $\mathbb{F}$ is the free group on countably many generators, then using the fact that $C^*(\mathbb{F})$ is surjectively universal, we see that the QWEP conjecture is equivalent to the statement that $C^*(\mathbb{F})$ is QWEP.  The main point of this note is to point to an a priori weaker equivalent statement of the QWEP conjecture:

\begin{thm}\label{main}
The QWEP conjecture is equivalent to the statement that $C^*(\mathbb{F})$ is elementarily equivalent to a QWEP C$^*$ algebra.
\end{thm}

Here, two C$^*$ algebras $A$ and $B$ are \emph{elementarily equivalent} if they have the same first-order theories in the sense of model theory.  (Here, we work in the signature for \emph{unital} C$^*$ algebras.)  

The next lemma is probably well known to the experts but since we could not locate it in the literature we include a proof here.
\begin{prop}\label{folklore}
Let $A$ be a C$^*$ algebra and $\omega$ a nonprincipal ultrafilter on some (possibly uncountable) index set.
\begin{enumerate}
\item Suppose that $A$ is a subalgebra of $B$ and that $B$ admits a u.c.p.\ map into $A^\om$ that restricts to the diagonal embedding of $A$.  Then $A$ is relatively weakly injective in $B$.
\item $A$ is relatively weakly injective in $A^\om$.
\end{enumerate}
\end{prop}

\begin{proof}
Part (1) is almost identical to the easy direction of \cite[Corollary 3.2(ii)]{Kirchberg} except there he works with the corona algebra instead of the ultrapower.  In any event, the proof is easy so we give it here:  if $\phi:B\to A^\om$ is a u.c.p.\ map restricting to the diagonal embedding of $A$, then the desired weak expectation $\psi:B\to A^{**}$ is given by $\psi(b)=\lim_\om b_n$, where $b=(b_n)^\bullet$ and the ultralimit is taken with respect to the ultraweak topology.  (2) follows immediately from (1) by taking a nonprincipal ultrafilter $\om'$ on a big enough index set so as to allow for an embedding $A^\om$ into $A^{\om'}$ that restricts to the diagonal embedding of $A$ into $A^{\om'}$.    
\end{proof}

Theorem \ref{main} follows immediately from the following:

\begin{prop}\label{axiom}
The set of C$^*$ algebras with QWEP forms an axiomatizable class.
\end{prop}

\begin{proof}
We use the semantic test for axiomatizability, namely we show that the class of QWEP algebras is closed under isomorphism, ultraproduct, and ultraroot.  Clearly the class of QWEP algebras is preserved under isomorphism.  To see that it is closed under ultraproducts, it suffices to note that it is closed under products \cite[Corollary 3.3(i)]{Kirchberg} and (obviously) closed under quotients.  To see that it is closed under ultraroots, we use the fact that $A$ is relatively weakly injective in its ultrapower (Proposition \ref{folklore}(2)) together with the fact that QWEP passes to relatively weakly injective unital subalgebras \cite[Corollary 3.3(iii)]{Kirchberg}.
\end{proof}

\begin{rmk}
Proposition \ref{axiom} is false with QWEP replaced by WEP:  in \cite{FAE} the authors show that the ultrapower of $\B(H)$ does not have WEP.
\end{rmk}

\begin{rmk}
Inductive limits of QWEP algebras are again QWEP (see \cite[Lemma 13.3.6]{BO}), so the class of QWEP algebras is $\forall\exists$-axiomatizable.
\end{rmk}

There is one other model-theoretic way to settle the QWEP conjecture; we refer the reader to \cite{KEP} for the definition of existential embeddings.

\begin{prop}
The QWEP conjecture is equivalent to the statement that there is a QWEP C$^*$ algebra $A$ containing $C^*(\mathbb{F})$ as a subalgebra such that the inclusion is an existential embedding of operator systems.
\end{prop}

\begin{proof}
Suppose that $A$ is as in the conclusion of the proposition.  Then there is a u.c.p.\  embedding $A\hookrightarrow C^*(\mathbb{F})^\om$ whose restriction to $C^*(\mathbb{F})$ is the diagonal embedding.  It follows that $C^*(\mathbb{F})$ is relatively weakly injective in $A$, whence it is itself QWEP by the aforementioned result of Kirchberg.
\end{proof}

The previous proposition appeared as \cite[Corollary 2.24]{KEP} but with QWEP replaced by WEP.

In \cite{omitting}, the authors ask whether or not every C$^*$ algebra is elementarily equivalent to a nuclear C$^*$ algebra.  It seems that the authors there were unaware of the fact that if their question had a positive answer, then the QWEP conjecture (and hence CEP) would also be settled.  Nevertheless, in the forthcoming manuscript \cite{many}, the authors settle this question in the negative by showing that neither $C^*_r(\mathbb{F})$ nor $\prod_\om M_n$ have nuclear models.

A question of Kirchberg, first appearing in print in \cite{ozawa}, asks something seemingly more modest than the QWEP conjecture:  is there an example of a non-nuclear C$^*$ algebra that has both WEP and the \emph{local lifting property} (LLP)?  Indeed, Kirchberg showed that the QWEP conjecture is equivalent to the statement that the LLP implies WEP.  Let us call the statement that there exists a non-nuclear C$^*$ algebra that has both WEP and LLP the \emph{weak QWEP conjecture}.

\begin{prop}
Let $A$ be either $C^*_r(\mathbb{F})$ or $\prod_\om M_n$.  If $A$ is elementarily equivalent to a C$^*$ algebra $B$ with LLP, then $B$ yields a positive solution to the weak QWEP conjecture.
\end{prop}

\begin{proof}
Since $A$ is QWEP (for the case of $C^*_r(\mathbb{F})$, see \cite[Proposition 3.3.8]{BO}), we have that $B$ is also QWEP by Proposition \ref{axiom}; since $B$ has LLP, we see that $B$ also has WEP \cite[Corollary 2.6(ii)]{Kirchberg}.  $B$ is not nuclear from the aforementioned result in \cite{many}.
\end{proof}

We end this note with something only tangentially related.  First a preparatory result:

\begin{prop}\label{cofinal}
Suppose that $A$ is a nonseparable C$^*$ algebra with a cofinal collection of separable subalgebras that have WEP.  Then $A$ has WEP.
\end{prop}

\begin{proof}
Suppose that $A\subseteq B$; we must show that $A$ is relatively weakly injective in $B$.  To see this, for each separable $C\subseteq A$ with WEP, there is a weak conditional expectation $\phi_C:B\to C^{**}\subseteq A^{**}$.  By taking a pointwise-ultraweak limit of $\phi_C$ as $C$ ranges over a cofinal family of separable subalgebras with WEP, we get a witness to the fact that $A$ is relatively weakly injective in $B$.
\end{proof}

The following first appeared as Theorem 3.1 of \cite{qe}.

\begin{cor}
The theory of unital C$^*$ algebras does not have a model companion, meaning that the class of existentially closed unital C$^*$ algebras is not axiomatizable.
\end{cor}

\begin{proof}
Suppose that $T$ is the model companion of the theory of C$^*$ algebras, so the models of $T$ are precisely the existentially closed C$^*$ algebras.  By \cite[Corollary 2.4]{KEP}, all models of $T$ have WEP.  Let $A$ be a model of $T$.  Then $A^\om$ has a cofinal collection of WEP separable subalgebras, namely the separable elementary substructures of $A^\om$.  By Proposition \ref{cofinal}, $A^\om$ has WEP, whence $A$ is subhomogeneous by \cite[Corollary 4.14]{FAE}.  In particular, $A$ is finite.  Since existentially closed C$^*$ algebras are purely infinite by \cite[Corollary 2.7]{KEP}, we have a contradiction.  
\end{proof}

The advantage of the previous proof over the one in \cite{qe} is that the latter proof invokes a serious result of Haagerup and Thjorbornsen \cite{HT}, while the above proof ultimately relies instead on the (fundamental but more elementary) work of Junge and Pisier \cite{jungepisier}.

%
%
%
%

\end{document}